\documentclass{article}
\usepackage[lang = british]{ems-jems}

\usepackage[shortlabels]{enumitem}
\usepackage{mathtools}
\usepackage{amsthm}
\usepackage{thmtools}
\usepackage{float}
\usepackage{caption}

\usepackage{graphicx}

\usepackage{pgf} 
\usepackage{multicol}
\usepackage{multirow}
\usepackage{geometry,mathtools}
\usepackage{pgfplots}
\usepackage{lipsum}
\usepackage{tikz}
\usepackage{color,soul}
\usepackage[T1]{fontenc}

\pgfplotsset{compat=1.15}
\usepackage{mathrsfs}
\usetikzlibrary{arrows}
\usetikzlibrary{arrows,calc}
\usetikzlibrary{decorations.pathmorphing}

\usepackage{hyperref}

\newtheorem{Thm}{Theorem}[section]
\newtheorem{Lem}[Thm]{Lemma}
\newtheorem{Prop}[Thm]{Proposition}
\newtheorem{Cor}[Thm]{Corollary}
\newtheorem{Conj}[Thm]{Conjecture}

\newtheorem{problem}[Thm]{Problem}

\theoremstyle{definition}
\newtheorem{Def}[Thm]{Definition}

\theoremstyle{remark}
\newtheorem{Rem}[Thm]{Remark}

\makeatletter
\def\moverlay{\mathpalette\mov@rlay}
\def\mov@rlay#1#2{\leavevmode\vtop{%
   \baselineskip\z@skip \lineskiplimit-\maxdimen
   \ialign{\hfil$\m@th#1##$\hfil\cr#2\crcr}}}
\newcommand{\charfusion}[3][\mathord]{
    #1{\ifx#1\mathop\vphantom{#2}\fi
        \mathpalette\mov@rlay{#2\cr#3}
      }
    \ifx#1\mathop\expandafter\displaylimits\fi}
\makeatother

\newcommand{\Pfin}{\mathfrak P_{\mathrm{fin}}}

\newcommand{\Pthree}{\mathfrak P_{3}}

\newcommand{\structA}{{\mathbb A}}
\newcommand{\structB}{{\mathbb B}}
\newcommand{\structC}{{\mathbb C}}
\newcommand{\structS}{{\mathbb S}}
\newcommand{\struct}[1]{\mathbb{#1}}
\newcommand{\arrC}{{\downarrow\!\overline{\cloC_3}}}

\newcommand{\NN}{{\mathbb N}}

\newcommand{\cloA}{{\mathcal A}}
\newcommand{\cloB}{{\mathcal B}}
\newcommand{\cloC}{{\mathcal C}}
\newcommand{\cloD}{{\mathcal D}}
\newcommand{\cloI}{{\mathcal I}}
\newcommand{\cloS}{{\mathcal S}}

\newcommand{\eqm}{{\equiv_{\mathrm{m}}}}
\newcommand{\leqm}{{\preceq_{\mathrm{m}}}}
\newcommand{\nleqm}{{\npreceq_{\mathrm{m}}}}
\newcommand{\leqcon}{{\leq_{\mathrm{Con}}}}
\newcommand{\Malcev}{\Sigma_{\operatorname{M}}}

\DeclareMathOperator{\pr}{pr}
\DeclareMathOperator{\CSP}{CSP}
\DeclareMathOperator{\Pol}{Pol}

\DeclareMathOperator{\Inv}{Inv}

\DeclareMathOperator{\GM}{GM}

\begin{document}

\title{Submaximal clones over a three-element set up to minor-equivalence}
\titlemark{Submaximal clones over a three-element set up to minor-equivalence}

\emsauthor{1}{Albert Vucaj}{A.~Vucaj}

\emsauthor{2}{Dmitriy Zhuk}{D.~Zhuk}

\emsaffil{1}{Institute of Discrete Mathematics and Geometry, Technische Universit\"at Wien, Wien, Austria
\email{albert.vucaj@tuwien.ac.at}}

\emsaffil{2}{Department of Algebra, Faculty of Mathematics and Physics, Charles University\\ Prague, Czechia
\email{zhuk.dmitriy@gmail.com}}



\begin{abstract}
We study clones
modulo minor homomorphisms, which are mappings from one clone to another preserving arities of operations 
and respecting permutation and identification of variables. 
Minor-equivalent clones satisfy the same sets of identities of the form $f(x_1,\dots,x_n)\approx g(y_1,\dots,y_m)$, also known as minor identities, and therefore share many algebraic properties.
Moreover, it was proved that the complexity of the $\operatorname{CSP}$ of a finite structure $\struct{A}$ only depends on the set of minor identities satisfied by the polymorphism clone of $\struct{A}$.
In this article we consider the poset that arises by considering all clones over a three-element set with the following order: we write $\cloC\ \leqm\ \cloD$ if there exists a minor homomorphism from $\cloC$ to $\cloD$. We show that the aforementioned poset has only three submaximal elements.
\end{abstract}

\maketitle

\section{Introduction}\label{sec:Intro}
In 1959 Janov and Mu\v{c}nik~\cite{3elem} proved that there exists a continuum of clones over a $k$-element set, for every $k \geq 3$. Thus, the goal to achieve a classification à la Post~\cite{Post} for clones over a three-element set seemed to falter. Subsequently, researchers in universal algebra focused on understanding particular aspects of clone lattices on finite domains. As far as concerns clones over $\{0,1,2\}$, remarkable results in this direction are the description of all \emph{maximal clones}~\cite{JabMaximal} and of all \emph{minimal clones}~\cite{Csakany84}. Moreover, it turned out that all the aforementioned maximal clones, with the sole exception of the clone of all linear operations, contain a continuum of subclones~\cite{DemetrHannak,Marchenkov}. More recently, a complete description of all clones of self-dual operations over a three-element set was provided~\cite{Zhuk15}. Note that this is a remarkable result since the clone of all self-dual operations, which we denote by $\cloC_3$, is one of the maximal clones over $\{0,1,2\}$; thus $\cloC_3$ is the first maximal clone besides the clone of all linear operations that has such description. In particular, $\cloC_3$ is the only maximal clone which has a full description of all its subclones, despite having continuum many of them. Another result that seems to be a setback in the research-line aimed at describing the entire lattice of clones over $\{0,1,2\}$ is the following: it is undecidable whether a given clone over a finite domain is \emph{finitely related}~\cite{mooreFinRelated}. 

One might still hope to classify all operation clones over finite sets up to some equivalence relation so that equivalent clones share many of the properties that are of interest in universal algebra.
Recently, Barto, Opr\v{s}al, and Pinsker~\cite{wonderland} introduced a weakening of the notion of clone homomorphism on the class of clones over a finite set, known in the literature as \emph{minor homomorphism}. We write $\cloC\ \leqm\ \cloD$ if there exist a minor homomorphism from $\cloC$ to $\cloD$, that is, a map preserving arities and taking minors, where a \emph{minor} of an operation $f$ is an operation obtain from $f$ by permuting its variables, identifying variables, or by adding dummy variables (see Definitions~\ref{def:minor} and~\ref{def:minorpreservingmaps}). Moreover, we write $\cloC\ \eqm\ \cloD$ if $\cloC\ \leqm\ \cloD$ and $\cloD\ \leqm\ \cloC$ and say that $\cloC$ and $\cloD$ are \emph{minor-equivalent}; by $\overline{\cloC}$ we denote the $\eqm$-class of $\cloC$, i.e., the class of all clones over some finite set which are minor-equivalent to $\cloC$. The relation $\leqm$ is a reflexive and transitive relation on the class of clones over a finite set, hence $\eqm$ is indeed an equivalence relation, thereby making the use of these suggestive symbols justified. Minor-equivalent clones satisfy the same sets of identities of a particular form, known as \emph{minor conditions} (see Section~\ref{sec:ourposet}). In a recent turn of events, it turned out that the complexity of the Constraint Satisfaction Problem of $\structA$ ($\CSP(\structA)$), where $\structA$ is a finite relational structure with finite signature, only depends on the set of  minor identities satisfied by $\Pol(\structA)$, i.e., by the \emph{polymorphism clone} of $\structA$. Moreover, the relation $\leqm$ preserves the complexity CSPs, for short: if $\Pol(\structA)\ \leqm\ \Pol(\structB)$ then there exists a log-space reduction from $\CSP(\structB)$ to $\CSP(\structA)$~\cite{wonderland}.

In this article we focus on the set of all clones over $\{0,1,2\}$ ordered with respect to $\leqm$. More precisely, we describe the submaximal elements of the poset $\mathfrak{P}_3 \coloneqq(\{\overline{\mathcal{C}} \mid \mathcal C\text{ is a clone over } \{0,1,2\}\};\leqm)$.
A full description of the subposet $\arrC$ which contains all the elements of $\Pthree$ which are smaller than $\overline{\cloC_3}$, with respect to $\leqm$, was provided in~\cite{BodirskyVucajZhuk}. From the latter description it follows that $\arrC$ is a countably infinite lattice. In the same article (\cite{BodirskyVucajZhuk}, Conjecture~6.2) it was conjectured that $\Pthree$ has exactly three submaximal elements, namely the $\eqm$-classes of the following three clones:
\begin{align*}
        \cloC_2 &\coloneqq \Pol\big((\{0,1\};\{(0,1),(1,0)\})\big),
        \\\cloC_3 &\coloneqq \Pol\big((\{0,1,2\};\{(0,1),(1,2),(2,0)\})\big),
        \\\cloB_2 &\coloneqq \Pol\big((\{0,1\};\{0\},\{1\},\{(0,1),(1,0),(1,1)\})\big).    
\end{align*}
Note that $\cloC_2$ and $\cloB_2$ are clones over $\{0,1\}$; we could equivalently consider the $\eqm$-classes of the following clones:
\begin{align*}
\cloC'_2 &\coloneqq\Pol\big((E_3;\{(0,1),(1,0)\})\big) \mbox{ and }\\\cloB'_2&\coloneqq\Pol\big((E_3;\{0\},\{1\},\{(0,1),(1,0),(1,1)\})\big),    
\end{align*} 
where $E_3\coloneqq\{0,1,2\}$.

\subsection*{Contributions} We give a positive answer to the aforementioned conjecture: we show that $\Pthree$ has exactly three submaximal elements, namely $\overline{\cloC_2}$, $\overline{\cloC_3}$, and $\overline{\cloB_2}$ (see Corollary~\ref{thm:main}). Our proof is of syntactic nature: we prove several statements that entail the existence of operations satisfying suitable identities in some clone $\cloC$, provided that $\cloC$ has certain operations. For example, we prove that if $\cloC$ is an idempotent clone over an $n$-element set, for some $n\geq 2$, such that $\cloC$ has a \emph{Mal'cev operation} and a \emph{cyclic operation of arity} $p$, for every prime $p\leq n$, then $\cloC$ has a \emph{majority operation} (see Lemma~\ref{lem:majorityWithKey}).
Statements of this form constitute results of independent interest in universal algebra.

\section{Preliminaries}\label{sec:preliminaries}
In this section we present notation, definitions, and some basic results from the literature which we are going to use throughout the article.

\subsection{A Galois connection for clones}
We denote the set $\{0,\dots,k-1\}$ by $E_k$. For $n \in {\mathbb N}$, we define 
\begin{align*}
    {\mathcal O}^{(n)}_k \coloneqq \{f\mid f\colon E_k^n\to E_k\}&& \text{and} && {\mathcal O}_k \coloneqq \bigcup_{n \in {\mathbb N}} {\mathcal O}^{(n)}_k.
\end{align*}
A \emph{clone over $E_k$} is a subset $\cloC$ of ${\mathcal O}_k$ which is closed under composition of operations and which contains all projections, i.e., operations of the form $\pr^{n}_{i}\colon (a_1,\dots,a_n)\mapsto a_i$, for all $a_1,\dots,a_n\in E_k$. If $F \subseteq {\mathcal O}_k$, then $\langle F\rangle$ denotes the \emph{clone generated by $F$}, i.e., the smallest clone that contains~$F$. 

An alternative way of describing a clone of operations is to specify the clone as the set all operations preserving a given set of relations. We say that an $n$-ary operation $f$ \emph{preserves a relation} $R$ on a finite set $A$ if, for every $a_1,\dots,a_n \in R$, it holds that $f (a_1,\dots,a_n) \in R$. In this case we also say that $R$ is \emph{invariant under} $f$. If $f$ preserves all the relations in $\Gamma$, we say that $f$ is a \emph{polymorphism of} $\Gamma$.

Let $\Gamma$ be a set of relations over $E_k$, for some $k\in\NN$. We define the set
\begin{align*}
    \Pol(\Gamma) \coloneqq \bigcup_{n\in {\mathbb N}}\Big\{f\colon E_k^n\to E_k \mid f \text{ preserves all the relations in } \Gamma\Big\}.
\end{align*}
We call the set $\Pol(\Gamma)$ \emph{the polymorphism clone of} $\Gamma$.

Analogously, for every set of operations $F$ over $E_k$, we define
\begin{align*}
    \Inv(F) \coloneqq \bigcup_{n\in {\mathbb N}}\Big\{R\subseteq E_k^n \mid R \text{ is invariant under every operation in } F\Big\}.
\end{align*}

\begin{Thm}[\cite{BoKaKoRo,Geiger}]\label{thm:PolInv}
Let $F$ be a set of operations over a finite set. The following equality holds: $\Pol(\Inv(F)) = \langle F \rangle$.
\end{Thm}

Let $\tau$ be a relational signature. A $\tau$-structure is a relational structure over the signature $\tau$.  
Let $\structA$ and $\structB$ be two relational $\tau$-structures. A map $h\colon A\to B$ is a \emph{homomorphism} if for every $R\in\tau$
\begin{equation*}
    \text{if } (a_1,\dots,a_n)\in R^\structA, \text{ then } (h(a_1),\dots,h(a_n))\in R^\structB.
\end{equation*}
We denote by $\operatorname{Hom}(\structA,\structB)$ the set of all homomorphisms from $\structA$ to $\structB$. We also write $\structA\to\structB$ if there exists a homomorphism from $\structA$ to $\structB$, and we say that $\structA$ and $\structB$ are \emph{homomorphically equivalent} if $\structA\to\structB$ and $\structB\to\structA$. An \emph{endomorphism of} $\structA$ is a homomorphism from $\structA$ to $\structA$. An \emph{isomorphism} between $\structA$ and $\structB$ is a bijective homomorphism $h$ such that the mapping $h^{-1}\colon B \to A$ that sends $h(x)$ to $x$ is a homomorphism, too. An \emph{automorphism of} $\structA$ is an isomorphism between $\structA$ and itself. A finite structure $\structA$ is called a \emph{core} if every endomorphism of $\structA$ is an automorphism. We say that $\structC$ is a \emph{core of} $\structA$ if $\structC$ is a core and $\structC$ is homomorphically equivalent to $\structA$. It is well known that every finite relational structure has a core which is unique up to isomorphism, thus it makes sense to speak about \emph{the} core of a relational structure.

In this article, we are going to consider polymorphism clones of relational structures in addition to polymorphism clones of a set of relations: we define $\Pol(A;\Gamma)\coloneqq\Pol(\Gamma)$.
For $n\geq 1$, we denote by $\structA^n$ the structure with the same signature $\tau$ as $\structA$ whose domain is $A^n$ such that for every $k$-ary $R\in \tau$, it holds that  $(\boldsymbol{a}_1,\dots,\boldsymbol{a}_k)$ is contained in $R^{\structA^n}$ if and only if it is contained in $R^\structA$ componentwise, i.e., $(a_{1j},\dots,a_{kj})\in R^\structA$ for every $1\leq j\leq n$.
Note that, equivalently, $\Pol(\structA)=\bigcup_{n\in\mathbb{N}}\operatorname{Hom}(\structA^n,\structA)$.

A \emph{primitive positive formula over $\tau$} is a first-order formula which only uses relation symbols in $\tau$, equality, conjunction and existential quantification. If $\structA$ is a $\tau$-structure and $\phi(x_1,\dots,x_n)$ is a $\tau$-formula with free-variables $x_1,\dots,x_n$, then 
$\{(a_1,\dots,a_n)\mid \structA \models \phi(a_1,\dots,a_n)\}$
is called \emph{the relation defined by $\phi$ in $\structA$}. 
In particular, if $\phi$ is primitive positive, then this relation is said to be \emph{pp-definable} in $\structA$. Given two relational structures $\structA$ and $\mathbb B$ on the same domain -- but with possibly different signatures -- we say that $\structA$ \emph{pp-defines} $\structB$ if every relation in $\structB$ is pp-definable in $\structA$.

\begin{Thm}[\cite{BoKaKoRo, Geiger}]\label{thm:InvPol}
Let $\structA$ be a finite relational structure. A relation $R$ has a pp-definition in $\structA$ if and only if $R\in\Inv(\Pol(\structA))$.
\end{Thm}
 
It is well known that all clones of operations over a fixed set $E_n$ form an algebraic lattice $\mathfrak{L}_n$ under set inclusion. The lattice operations are defined as follows: $\cloC\wedge\cloD\coloneqq \cloC \cap \cloD$ and $\cloC\vee\cloD\coloneqq \langle\cloC \cup \cloD\rangle$. The top-element of $\mathfrak{L}_n$ is the clone ${\mathcal O}_{n}$, its bottom-element is the clone of all projections over $E_n$, which we denote by $\mathcal{P}_n$. A celebrated result, due to Post~\cite{Post}, is the full description of $\mathfrak{L}_2$. Janov and Mu\v{c}nik~\cite{3elem} proved that $\mathfrak{L}_n$ has a continuum of elements, for every $n \geq 3$.

Moreover, Theorem~\ref{thm:InvPol} underlines that pp-definability among relational structures translates to inclusion of the correspondent polymorphism clones.

\begin{Thm}[\cite{BoKaKoRo,Geiger}]\label{cor:DefAndCloneInclusion}
Let $\structA$ and $\structB$ be structures on the same finite set $A$. Then $\structA$ pp-defines $\structB$ if and only if $\Pol(\structA)\subseteq\Pol(\structB)$.
\end{Thm}

\subsection{The pp-constructability poset}\label{sec:ourposet}
In this section we briefly introduce the pp-constructability poset -- the main object of study in this article. For this purpose, we first define the notions of pp-constructability and minor homomorphism.

Let $\structA$ and $\structB$ be finite relational structures. We say that $\structB$ is a \emph{pp-power} of $\structA$ if it is isomorphic to a structure with domain $A^n$, for some $n\geq 1$, whose relations are pp-definable from $\structA$. Notice that a $k$-ary relation on $A^n$ is regarded as a $kn$-ary relation on $A$.

\begin{Def}
Let $\structA$ and $\structB$ be finite relational structures. We say that $\structA$ \emph{pp-constructs} $\structB$, and write $\structA\ \leqcon\ \structB$, if $\structB$ is homomorphically equivalent to a pp-power of $\structA$. We also write $\structA\equiv_{\mathrm{Con}}\structB$ if $\structA\ \leqcon\ \structB$ and $\structB\ \leqcon\ \structA$.
\end{Def}
The question that naturally arises now is how pp-constructability among relational structures translates in terms of the correspondent
polymorphism clones: in particular, whether there is a Galois connection that would lead us to a result of the same flavour of Corollary~\ref{cor:DefAndCloneInclusion}.

\begin{Def}\label{def:minor}
Let $f$ be any $n$-ary operation, and let $\sigma$ be a map from $E_n$ to $E_r$. We denote by $f_\sigma$ the following $r$-ary operation 
\begin{equation*}
f_\sigma(x_0,\dots,x_{r-1}) \coloneqq f(x_{\sigma(0)},\dots,x_{\sigma(n-1)}).
\end{equation*}
Any operation of the form $f_\sigma$, for some $\sigma \colon E_m\to E_n$, is called a \emph{minor} of $f$.
\end{Def}
A \emph{minor identity} is a formal expression of the form 
\[\forall x_1,\dots,x_r (f(x_{\sigma(0)},\dots,x_{\sigma(n-1)}) = g(x_{\pi(0)},\dots,x_{\pi(m-1)})),\] 
where $f$ and $g$ are function symbols and $\sigma\colon E_n \rightarrow E_r$ and $\pi\colon E_m \rightarrow E_r$ are some maps; in this case we write $f_\sigma \approx g_\pi$. A \emph{minor condition} is a finite set of minor identities.
\begin{Def}\label{def:minorpreservingmaps}
Let $\cloA$ and $\cloB$ be clones and let $\xi \colon \cloA \rightarrow \cloB$ be a mapping that preserves arities. We say that $\xi$ is a \emph{minor homomorphism} if
\begin{equation*}
\xi(f_\sigma) = \xi(f)_\sigma
\end{equation*}
for any $n$-ary operation $f \in \cloA$ and $\sigma\colon E_n \rightarrow E_r$.
\end{Def}

We say that a set of operations $F$ \emph{satisfies} a minor condition $\Sigma$, and write $F \models\Sigma$, if every function symbol in $\Sigma$ can be mapped to an operation in $F$ such that, for every $f_{\sigma}\approx g_{\pi}$ in $\Sigma$, the equality $f_{\sigma}^F = g_{\pi}^F$ holds for every evaluation of the variables. Moreover, we say that an operation $f$ satisfies a minor condition $\Sigma$ if $\{f\}\models\Sigma$.

Next we are going to define some examples of minor conditions, which we will use in Section~\ref{sec:subOn3}.

\begin{Def}\label{def:symmMinorConditions}
We define the following minor conditions:
\begin{itemize}

\item We call \emph{cyclic identity of arity} $p$, for some $p\geq 2$, the following identity
\begin{equation}\tag{$\Sigma_p$}
    c(x_1,x_2,\dots,x_p)\approx c(x_2,\dots,x_p,x_1).
\end{equation}

\item We call \emph{quasi minority} the following minor condition:
\begin{equation*}m(x,y,y)\approx m(y,x,y)\approx m(y,y,x)\approx m(x,x,x).\end{equation*}

\item We call \emph{quasi Mal'cev} the following minor condition:
\begin{equation}\tag{$\Sigma_{\operatorname{M}}$}m(x,y,y)\approx m(y,y,x)\approx m(x,x,x).\end{equation}

\item We call \emph{quasi majority} the following minor condition:
\begin{equation*}m(x,y,y)\approx m(y,x,y)\approx m(y,y,x)\approx m(y,y,y).\end{equation*}

\item We call $n$\emph{-ary symmetric condition} the minor condition that consists of all identities of the form
\begin{equation*}\tag{$\operatorname{FS}(n)$}
    f(x_1, x_2,\dots, x_n) \approx f(x_{\pi(1)}, x_{\pi(2)},\dots,x_{\pi(n)}),
\end{equation*}
where $\pi$ is a permutation of the set $\{1,2,\dots,n\}$.

\item  We call $n$\emph{-ary totally symmetric condition} the minor condition that consists of all identities of the form 
\begin{equation*}\tag{$\operatorname{TS}(n)$}
   f(x_{i_1},x_{i_2},\dots,x_{i_n}) \approx f(x_{j_1},x_{j_2},\dots,x_{j_n}), 
\end{equation*}
whenever $\{i_1, i_2,\dots,i_n\} = \{j_1,j_2,\dots,j_n\}$.

\item  We call $n$\emph{-ary generalized minority condition},
where $n\ge 3$ is odd,
the minor condition that consists of all the identities from  $\operatorname{FS}(n)$
and 
\begin{equation*}\tag{$\operatorname{GM}(n)$}
   f(x,x,x_3,x_4,\dots,x_n) \approx 
   f(y,y,x_3,x_4,\dots,x_n).
\end{equation*}

\item We call \emph{weak near-unanimity condition of arity} $n\geq 3$ the following minor condition 
\begin{equation}\tag{$\operatorname{WNU}(n)$}
    w(x,\dots,x,y)\approx w(x,\dots,x,y,x) \approx\ldots\approx w(y,x,\dots,x).
\end{equation}

\item We call \emph{quasi near-unanimity condition of arity} $n\geq 3$ the following minor condition
\begin{align}
   w(x,\dots,x,y)&\approx w(x,\dots,x,y,x)\approx\notag
   \\ &\approx\ldots\approx w(y,x,\dots,x)\approx w(x,\dots,x).\tag{$\operatorname{QNU}(n)$}
\end{align}
\end{itemize}
\end{Def}

A $k$-ary operation $f$ is a \emph{quasi Mal'cev operation} if it satisfies the minor condition $\Malcev$; we adopt an analogous convention for all other conditions introduced so far. A \emph{Mal'cev operation} is an idempotent quasi Mal'cev operation; we adopt a similar convention in defining a \emph{minority operation}, a \emph{generalized minority operation}, a \emph{majority operation}, and a \emph{near-unanimity operation}. If $f$ is a totally symmetric operation of arity $n$ we also write $f(\{x_1, x_2,\dots,x_n\})$ instead of $f(x_1, x_2,\dots,x_n)$. Note that every totally symmetric operation is symmetric. The other implication does not hold: the majority operation over $\{0,1\}$ is symmetric but not totally symmetric. The multiplicity of variables plays a role in the definition of a symmetric operation: as a matter of fact, $f(x,x,y)$ need not be equal to $f(x,y,y)$.

In Remark~\ref{Ex:cyclesAndCyclic} we show an example of a minor condition that is not satisfied by the polymorphisms of a certain structure: the polymorphism clone of a directed cycle of length $p$ does not satisfy the cyclic identity of arity $p$. This is an easy observation that will nevertheless come in handy later in this article, more precisely, in the proof of Theorem~\ref{thm:splittingCycles}.

\begin{Rem}\label{Ex:cyclesAndCyclic}
Consider the relational structure $\structC_p = (E_p;R^{\structC_p})$, where
\[R^{\structC_p}\coloneqq \{(0,1),\dots,(p-2,p-1),(p-1,0)\}.\]

It is easy to see that, for every prime $p$, it holds that $\Pol(\structC_p)$ does not satisfy the cyclic identity of arity $p$, i.e., $\Pol(\structC_p)\not\models\Sigma_p$.
Indeed, suppose that there exists a polymorphism $f$ of $\structC_p$ satisfying $\Sigma_p$, then
\[f(0,\dots,p-1) = f(1,\dots,p-1,0) = a, \text{ for some } a\in E_p.\]
It would follow that $(a,a)\in R^{\structC_p}$, a contradiction.
\end{Rem}
The connection between pp-constructability and minor homomorphisms is given by the following theorem. 

\begin{Thm}[\cite{wonderland}]\label{thm:barto}
Let $\structA$, $\structB$ be finite relational structures and $\cloA=\Pol(\structA)$, $\cloB=\Pol(\structB)$. The following are equivalent:
\begin{enumerate}
    \item There exists a minor homomorphism from $\cloA$ to $\cloB$;
    \item $\structA$ pp-constructs $\structB$;
    \item if $\cloA$ satisfies a minor condition $\Sigma$, then $\cloB\models\Sigma$.
\end{enumerate}
\end{Thm}

We would like to remark that, unfortunately, the latter theorem does not yield a Galois connection for clones: Barto, Opr\v{s}al, and Pinsker also provide a semantic characterization of minor homomorphisms, by introducing a new operator called reflection~\cite{wonderland}. However, a reflection of a clone need not contain projections or be closed under composition; thus, the reflection of a clone is, in general, not a clone. What we obtain by taking a reflection of a clone is known in the literature as minion; they play a crucial role in Promise Constraint Satisfaction Problems~\cite{BBKO}.

\begin{Cor}[\cite{wonderland}]\label{cor:coresAndIdentities}
Let $\structA$ be a finite relational structure, let $\structC$ be the core of $\structA$ and let $\structC^c$ be the expansion of $\structC$ by all unary relations $\{a\}_{a\in C}$. Then:
\begin{enumerate}
    \item $\structA\equiv_{\mathrm{Con}}\structC \equiv_{\mathrm{Con}}\structC^c$;
    \item for every minor condition $\Sigma$, $\Pol(\structA)\models\Sigma$ if and only if $\Pol(\structC^c)\models\Sigma$.
\end{enumerate}
\end{Cor}

We write $\cloA\ \leqm\ \cloB$ if there exists a minor homomorphism $\xi\colon \cloA \rightarrow \cloB$, and we denote by $\eqm$ the equivalence relation where
$\cloA\ \eqm\ \cloB$ if $\cloA \ \leqm\ \cloB$ and $\cloB \ \leqm\ \cloA$. Note that $\cloA \ \subseteq\ \cloB$ implies $\cloA \ \leqm\ \cloB$. Moreover, we denote by $\overline{\cloA}$ the $\eqm$-class of $\cloA$, that is
\[\overline{\cloA}\coloneqq\{\cloC\mid \cloC \text{ is  a clone over some finite set and } \cloC\ \eqm\ \cloA\}\]
and we write $\overline{\cloA}\ \leqm\ \overline{\cloB}$ if and only if $\cloA\ \leqm\ \cloB$.

We finally define the following posets:
\begin{align*}
\Pfin &\coloneqq(\{\overline{\mathcal{C}} \mid \mathcal C\text{ is a clone over a finite set}\};\leqm);
\\\mathfrak{P}_n &\coloneqq(\{\overline{\mathcal{C}} \mid \mathcal C\text{ is a clone over } E_n\};\leqm)
\end{align*}
and call them the \emph{pp-constructability poset} and the \emph{pp-constructability poset restricted to clones over an $n$-element set}, respectively.
Note that a full description of $\mathfrak{P}_2$ was presented in \cite{albert}, and this article entirely focuses on $\mathfrak{P}_3$. 

A \emph{constant operation} of arity $n$ is an operation $c^{(n)}$ defined as follows
\[c^{(n)}(x_1,\dots,x_n)\coloneqq c\]
where $c\in E_m$, for some $n,m\geq 1$; if $n=1$, we simply write $c$ to denote the unary constant operation $c^{(1)}$. We want to remark that if $\cloC$ has a constant operation $c^{(n)}$, for some $n$, then it has a constant operation for every arity. It is easy to show that $\overline{\langle 0\rangle}$ is the top-element of $\Pthree$. Indeed, for every clone over $\{0,1,2\}$ the map that sends every $n$-ary operation to $0^{(n)}$ is minor-preserving.

\subsection{The unique coatom}\label{sec:coatom}
In this section we prove that both the posets $\Pfin$ and $\Pthree$ have exactly one coatom, -- i.e., a unique maximal element -- namely $\overline{\cloI_2}$, that is, the $\eqm$-class of the clone of all idempotent operations over $\{0,1\}$.

Let us define the following relational structures
\begin{align*}
    \structC_1 &\coloneqq (\{0\};\{(0,0)\});
    \\\mathbb{I}_n &\coloneqq(E_n;\{0\},\dots,\{n-1\}), \text{ for every } n\geq 2.
\end{align*}

\begin{Prop}\label{prop:idempotent}
For every finite relational structure $\structA$ exactly one of the following holds: either $\structC_1\ \leqcon\ \structA$ or $\structA\ \leqcon\ \mathbb{I}_2$.
\end{Prop}
\begin{proof}
Let $\structA$ be a relational structure and let $\structB$ be its core expanded by all unary relations.
By Corollary~\ref{cor:coresAndIdentities} it holds that $\structC_1$ pp-constructs $\structA$ if and only if $\structC_1$ pp-constructs $\structB$ and $\structA$ pp-constructs $\struct{I}_2$ if and only if $\structB$ pp-constructs $\struct{I}_2$. Thus, we are going to prove the claim for $\structB$. Let $B=\{b_0,\dots,b_{n-1}\}$ be the domain of $\structB$.
If $n=1$, then it is straightforward to see that $\structC_1\ \leqcon\ \structB$. Let us assume that $n>1$, we need to show that $\structB\ \leqcon\ \mathbb{I}_2$. Consider the pp-power $\structS\coloneqq(\{b_0,\dots,b_{n-1}\};O,I)$ of $\structB$, where $O$ and $I$ are the unary relations defined by the formulae $O(x)\coloneqq (x=b_0)$ and $I(x)\coloneqq (x=b_1)$, respectively. Let us define the maps $g\colon \structS\to\struct{I}_2$ that maps $b_0$ to 0 and every other element to 1 and $h\colon\struct{I}_2\to\structS$ that maps 0 to $b_0$ and 1 to $b_1$. It is straightforward to check that $g$ and $h$ are homomorphisms. Thus $\struct{I}_2$ and $\structS$ are homomorphically equivalent and $\structB$ pp-constructs $\struct{I}_2$.
\end{proof}

Note that, for every $n\geq 2$, $\Pol(\structC_1)\ \nleqm\ \Pol(\struct{I}_n)$, since $\Pol(\structC_1)$ satisfies the minor identity $f(x)\approx f(y)$, while $\Pol(\struct{I}_n)$ does not. Therefore, for every $n\geq 2$, the structure $\structC_1$ does not pp-construct $\struct{I}_n$.

\begin{Prop}\label{prop:collapseIdemp}
For every $n \geq 2$ it holds that $\mathbb{I}_2 \equiv_{\mathrm{Con}} \mathbb{I}_n$.
\begin{proof}
First, we show that $\mathbb{I}_2\ \leqcon\ \mathbb{I}_n$. Consider the structure $\structS\coloneqq (\{0,1\}^n;\Phi_0,\dots,\Phi_{n-1})$ where each $\Phi_i$ is a relation defined as follows:
\begin{align*}
    \Phi_i\coloneqq \{(x_0,\dots,x_{n-1})\mid (x_i = 1) \wedge \bigwedge_{j \in E_n\setminus\{i\}} (x_j = 0)\}.
\end{align*}
Let us denote by $\boldsymbol{e}_i\in\{0,1\}^n$ the tuple that has a 1 in the $i$-th coordinate and 0s elsewhere. The maps \begin{align*}
    g\colon i\mapsto \boldsymbol{e}_i && h\colon \boldsymbol{x}\mapsto\begin{cases}
    i & \mbox{if } \boldsymbol{x} = \boldsymbol{e}_i,
    \\ 0 & \mbox{otherwise}.
    \end{cases}
\end{align*}
are respectively homomorphisms from $\mathbb{I}_n$ to $\structS$ and from $\structS$ to $\mathbb{I}_n$. This proves that $\mathbb{I}_2\ \leqcon\ \mathbb{I}_n$. The other inclusion follows from Proposition~\ref{prop:idempotent} and the fact that $\structC_1$ does not pp-construct $\struct{I}_n$, which is observed in the paragraph preceding
Proposition~\ref{prop:collapseIdemp}.
\end{proof}
\end{Prop}

Let us denote $\Pol(\struct{I}_n)$ by $\cloI_n$, for every $n\geq 2$. It follows from Proposition~\ref{prop:collapseIdemp} and Theorem~\ref{thm:barto} that $\cloI_2\ \eqm\ \cloI_3$. Thus, $\overline{\cloI_2}$ is the unique coatom in $\Pthree$. In Section~\ref{sec:subOn3} we classify all elements covered by $\overline{\cloI_2}$ in $\Pthree$.

\subsection{Auxiliary theorems}
Here, we present two results that are going to serve as auxiliary statements in Section~\ref{sec:subOn3}. We would like to point out that both of the statements we present in this section are known in the literature. However, in its original form, Theorem~\ref{thm:splittingMalcev} is not formulated in terms of pp-constructability. Similarly, although it is not found with this formulation anywhere in the literature, Theorem~\ref{thm:splittingCycles} can be obtained with a little work from Lemma~6.8 in \cite{cyclesBodirskyStrakeVucaj}. We therefore prefer to include two relatively simple and self-contained proofs with the aim of helping the reader. 
We define the structure: $\structB_2\coloneqq (\{0,1\};\{0\},\{1\},\{(0,1),(1,0),(1,1)\})$. 

\begin{Thm}[\cite{oprsal18}, Proposition 7.7]\label{thm:splittingMalcev}
Let $\structA$ be a finite structure. Then $\structA$ pp-constructs $\structB_2$ if and only if $\Pol(\structA)$ does not satisfy $\Sigma_{\operatorname{M}}$.
\end{Thm}
\begin{proof}
Note that by Corollary~\ref{cor:coresAndIdentities} it is sufficient to prove the claim for a core expanded by all unary relations; we therefore assume $\structA$ to be such a structure.
Suppose that $\structA$ pp-constructs $\structB_2$ and that $\Pol(\structA)$ satisfies $\Malcev$. By Theorem~\ref{thm:barto} it follows that $\Pol(\structB_2)\models\Malcev$. Hence, there exists a ternary operation $m\in\Pol(\structB_2)$ such that $m(0,1,1)=m(1,1,0)=m(0,0,0)=0$ (note that $m$ must preserve $\{0\}$). Since $m$ is a polymorphism of $\structB_2$, we obtain that $(0,0)$ is a tuple in a relation of $\structB_2$, a contradiction. For the other direction,
assume that 
$\Pol(\structA)$ does not satisfy
$\Sigma_{\operatorname{M}}$.
Our goal is to build a pp-power $\structC$ of $\structA$ that is homomorphically 
equivalent to $\structB_2$.
The domain of $\structC$ is $A^{|A|^{2}}$ and the elements of 
$\structC$ can be interpreted as operations
$A^{2}\to A$.
The structure $\structC$ has a binary relation $R^{\structC}$ and two unary relations $\{\pr^2_2\}$ and $\{\pr^2_1\}$; the relation $R^{\structC}$ is defined as
the set of all pairs $(f,g)$ such that 
there exists a ternary operation $w\in\Pol(\structA)$
such that $w(x,x,y) = f(x,y)$ and $w(y,x,x) = g(x,y)$.

Let us show that any $s$-ary operation $u\in\Pol(\structA)$
preserves $R^{\structC}$.
Consider some tuples $(f_1,g_1),\dots,(f_s,g_s)$ from $R^{\structC}$.
By definition, for every $i$, there exists 
$w_i\in\Pol(\structA)$ such that 
$w_i(x,x,y) = f_i(x,y)$ and $w_i(y,x,x) = g_i(x,y)$.
The operation $u$ applied to the considered tuples 
gives us 
the tuple $(f,g)$ where 
\begin{align*}
    f(x,y) &= u(f_1(x,y),\dots,f_s(x,y)), \text{and}
    \\g(x,y) &= u(g_1(x,y),\dots,g_s(x,y)).
\end{align*}
Since
$f(x,y) = w(x,x,y)$ and 
$g(x,y) = w(y,x,x)$ for 
\[w(x,y,z) \coloneqq u(w_1(x,y,z),\dots,w_s(x,y,z)),\]
we obtain that $(f,g)$ is in $R^{\structC}$. By Theorem \ref{cor:DefAndCloneInclusion}, 
$R^{\structC}$ is pp-definable over $\structA$. Since every $u\in\Pol(\structA)$ is idempotent, $u$ preserves two unary relations $\{\pr^2_2\}$ and $\{\pr^2_1\}$; thus both $\{\pr^2_2\}$ and $\{\pr^2_1\}$ are pp-definable over $\structA$.

It remains to define homomorphisms 
from $\structB_2$ to $\structC$ and back.
The homomorphism $h\colon\structB_2\to\structC$
maps $0$ to $\pr^2_2$ and
maps $1$ to the first projection $\pr^2_1$. 
Since all projections are in $\Pol(\structA)$ 
we have the tuples $(\pr^2_1,\pr^2_2)$, $(\pr^2_1,\pr^2_1)$, and $(\pr^2_2,\pr^2_1)$ in $R^{\structC}$. Hence $h$ is a
homomorphism.
The homomorphism $h'\colon\structC\to\structB_2$
maps $\pr^2_2$ to $0$ and all the remaining values to 
$1$. Note that the tuple $(\pr^2_2,\pr^2_2)$ is not in $R^{\structC}$, otherwise there would exist an operation $w\in\Pol(\structA)$ such that $\pr^2_2(x,y) = w(x,x,y) = w(y,x,x)=\pr^2_2(y,y)=w(y,y,y)$; this cannot hold since we assumed that $\Pol(\structA)$ does not satisfy $\Sigma_{\operatorname{M}}$.
This proves that $h'$ is a homomorphism. Thus, $\structA\ \leqcon\ \structB_2$ as desired.\end{proof}

Recall the relational structure $\structC_p$ introduced in Remark~\ref{Ex:cyclesAndCyclic}.

\begin{Thm}[\cite{cyclesBodirskyStrakeVucaj,DimaStrongSubalg}]\label{thm:splittingCycles}
Let $\structA$ be a finite structure. For every prime $p$, $\structA$ pp-constructs $\structC_p$ if and only if $\Pol(\structA)$ does not satisfy $\Sigma_p$.
\end{Thm}
\begin{proof}
Suppose that $\structA$ pp-constructs $\structC_p$, for some prime $p$, and that $\Pol(\structA)$ satisfies $\Sigma_p$; by Theorem~\ref{thm:barto} we obtain that $\Pol(\structC_p)$ satisfies $\Sigma_p$, too. This would lead to a contradiction, as shown in Remark~\ref{Ex:cyclesAndCyclic}.
For the other direction, assume that $\Pol(\structA)$ does not satisfy the cyclic identity $\Sigma_p$.
Let us build a pp-power $\structB$  of $\structA$ which is homomorphically equivalent to $\structC_p$.
The domain of 
$\structB$ is $A^{|A|^{p}}$ and every element of the domain can be viewed as an operation $A^{p}\to A$. The structure $\structB$ has only one 
binary relation $R^{\structB}$
which is defined as follows:
it consists of all pairs 
$(f,g)$ such that 
$f,g\in\Pol(\structA)$ and 
$g(x_1,\dots,x_p) = f(x_2,\dots,x_{p},x_1)$
(here we interpret operations as tuples). 

Let us partition all the operations 
of $\Pol(\structA)$ of arity $p$ into equivalence classes 
such that two operations are equivalent if 
one can be obtained from another by a cyclic shift of the variables.
Choose one representative from each class
and denote the obtained set of operations 
by $F_{0}$.
Let us define 
\[F_{i}=\{g\mid \exists f\in F_0\colon g(x_1,\dots,x_p) = 
f(x_i,\dots,x_p,x_1,\dots,x_{i-1})\}.\]
Thus, operations from $F_{i}$ 
are obtained from 
operations from $F_{0}$ by the corresponding cyclic shift of the variables.
If $F_{i}\cap F_{j}\neq\varnothing$ for some $i\neq j$, 
then there is an operation that stays the same after some cyclic shift of the variables. Since $p$ is prime, we derive that this operation is cyclic, which contradicts our assumption.

Thus, we assume that 
$F_{i}\cap F_{j}=\varnothing$ whenever $i\neq j$.
Let us show that any operation $w\in\Pol(\structA)$ preserves 
$R^{\structB}$. 
Consider tuples 
$(f_1,g_1),\dots,(f_s,g_s)$ from $R^{\structB}$.
Then $w$ applied to these tuples gives us 
the pair $(f,g)$, where
\begin{align*}
    f(x_1,\dots,x_p) &= w(f_1(x_1,\dots,x_p),\dots,f_s(x_1,\dots,x_p)), \text{and}
    \\g(x_1,\dots,x_p) &= w(g_1(x_1,\dots,x_p),\dots,g_s(x_1,\dots,x_p)).
\end{align*}
Hence, $f$ and $g$ are from $\Pol(\structA)$.
Moreover, it holds that 
$g(x_1,\dots,x_p) = f(x_2,\dots,x_p,x_1)$.
Thus, we showed that $w$ preserves $R^{\structB}$. By Theorem \ref{cor:DefAndCloneInclusion}, 
$R^{\structB}$ is pp-definable over $\structA$. Thus, we proved that $\structB$ is a pp-power of $\structA$.

It remains to show that
$\structB$ and $\structC_p$ are homomorphically equivalent.
The homomorphism $h\colon\structB\to\structC_p$
just sends operations from 
$F_{i}$ to $i$, for every $i$.
To define the homomorphism $h'\colon\structC_p\to\structB$
we just choose some operation $f_0\in F_{0}$, then 
consider all its cyclic shifts 
$f_{i}\in F_{i}$. Then $h'$ sends each $i$ to $f_{i}$.
Thus, $\structA$ pp-constructs $\structC_p$.
\end{proof}

We would like to mention that the proofs in this section can be considered rather standard, since in the literature the auxiliary structures that we define in Theorems~\ref{thm:splittingMalcev} and~\ref{thm:splittingCycles} -- denoted by $\structC$ and $\structB$, respectively -- are known as the free structure of $\Pol(\structA)$ generated by $\structB_2$ and the free structure of $\Pol(\structA)$ generated by $\structC_p$, respectively. To keep the proofs simple and self-contained we refrain from defining free structures and refer the interested reader to \cite{BBKO}.

\section{Submaximal elements of $\mathfrak P_3$}\label{sec:subOn3}
In this section we prove that $\overline{\cloC_2}$, $\overline{\cloC_3}$, and $\overline{\cloB_2}$ are the only submaximal elements of $\mathfrak{P}_3$. 
In particular, we show that if $\cloS$ is an idempotent clone over $\{0,1,2\}$ such that 
\begin{equation*}\tag{$\spadesuit$}\label{eq:submax}
    \cloS\ \nleqm\ \cloC_2,\ 
    \cloS\ \nleqm\ \cloC_3, \text{ and }
    \cloS\ \nleqm\ \cloB_2
\end{equation*}
then there exists a minor homomorphism from $\cloI_2$ to $\cloS$, i.e., $\cloI_2\ \leqm\ \cloS$. In order to prove this, we show that every idempotent clone over $E_3$ satisfying \eqref{eq:submax} has a generalized minority of arity $k$, for every odd $k\geq 3$ -- we are going to define such operations in Section~\ref{sec:oddition} -- and a totally symmetric operation of arity $n$, for every $n\geq 2$. Recall that we say that an $n$-ary operation $f$ is totally symmetric if it satisfies the condition $\operatorname{TS}(n)$ from Definition \ref{def:symmMinorConditions}.
Also, note that we can reduce to the case where we only consider idempotent clones: in fact, every clone $\cloS=\Pol(\structS)$ is minor-equivalent to some idempotent clone, that is, $\cloS\ \eqm\ \Pol(\structS')$ where $\structS'$ is the core of $\structS$ expanded by all unary relations (see Corollary~\ref{cor:coresAndIdentities}).
As a first step, we want to prove that every idempotent clone $\cloS$ over $E_3$ satisfying \eqref{eq:submax} has a majority operation. For this purpose, we introduce some more notions and terminology concerning relations. In particular, we are going to consider essential and critical relations.

\begin{Def}
Let $R$ be an $n$-ary relation on a finite set $A$. We say that 
\begin{itemize}
    \item 
    $R$ is an \emph{essential relation} if it cannot be represented by a quantifier-free conjunctive formula over relations 
    of smaller arities. Moreover, a tuple $(a_1,\dots,a_n)\in A^n\setminus R$ is \emph{essential for} $R$ if for every $i\in \{1,\dots,n\}$ there exists $b$ such that \[(a_1,\dots,a_{i-1},b,a_{i+1},\dots,a_n)\in R.\]
        We denote by $\operatorname{Ess}(R)$ the set of all essential tuples for $R$.
        It is not hard to see that a relation has an essential tuple if and only if it is essential~\cite{MinimalClones,Zhuk15}.
    \item $R$ is \emph{critical} if it is essential and there do not exist relations $R_1,\dots,R_m$ pp-definable from $R$ and different from $R$ such that $R = R_1\cap \dots \cap R_m$.
\end{itemize}
\end{Def}

Notice that originally in~\cite{KearnesSzendrei12} a critical relation was defined to be a directly indecomposable and completely $\cap$-irreducible relation in a relational clone; our definition of critical relation relates to the original one as follows: a relation is critical  if and only if it is critical -- in the sense of \cite{KearnesSzendrei12} -- in the relational clone generated by the relation.
As it was shown in \cite{ZhukKeyCritical}, every critical 
relation has internal symmetries, and relations having these symmetries are called key relations.
Thus, every critical relation is a key relation (see \cite{ZhukKeyCritical}, Lemma 2.4).
In this section we will use a classification
of key relations 
preserved by a weak unanimity operation 
from \cite{ZhukKeyCritical} but to avoid additional 
notations we formulate it for critical relations.

To formulate the classification we will need the notion of block of a relation $R$ over a finite domain $A$. We denote by $\tilde{R}$ the relation $R\cup \operatorname{Ess}(R)$. Again following \cite{ZhukKeyCritical}, we define a graph $\mathbb{G}_{\tilde{R}}\coloneqq(\tilde{R};E)$ as follows: if $\boldsymbol{a}$, $\boldsymbol{b}\in \tilde{R}\subseteq A^n$, then we have $(\boldsymbol{a},\boldsymbol{b})\in E$ if and only if $\boldsymbol{a}$ and $\boldsymbol{b}$ differ just in one element, i.e., there exists a unique $i\in\{1,\dots,n\}$ such that $a_i\neq b_i$. A \emph{block} of $R$ is a connected component of $\mathbb{G}_{\tilde{R}}$. A block is called \emph{trivial} if it only contains tuples from $R$.

\begin{Thm}[c.f. \cite{ZhukKeyCritical}, Theorem 3.11]\label{thm:DimaKeyCritical}
Let $R$ be a critical relation of arity $n\geq 3$, preserved by a Mal'cev operation. Then
\begin{itemize}
    \item Every block of $R$ equals $B_1 \times \cdots\times B_n$, for some $B_1,\dots,B_n\subseteq A$.
    \item For every nontrivial block $\textbf{B}\coloneqq B_1 \times \cdots\times B_n$ of $R$, the intersection $R\cap\textbf{B}$ can be defined as follows: there exists an abelian group $(G;+,-,0)$ whose order is a power of a prime, and surjective mappings $\phi_i\colon B_i\to G$, for $i=1,2,\dots,n$ such that
    \[R\cap\textbf{B} = \{(x_1,\dots,x_n)\mid \phi_1(x_1) + \phi_2(x_2) +\ldots +\phi_n(x_n) = 0\}.\]
\end{itemize}
\end{Thm}

\begin{Thm}[\cite{BP}]\label{thm:BakerPixley}
Let $\cloC$ be an idempotent clone over a finite set. Then, for every $k\geq 2$, the following are equivalent:
\begin{itemize}
    \item $\cloC$ has a near-unanimity operation of arity $k+1$;
    \item every $(k+1)$-ary relation in $\Inv(\cloC)$ can be obtained as a conjunction of relations of arity $k$ in $\Inv(\cloC)$.
\end{itemize}
\end{Thm}

We want to remark that, if an idempotent clone $\cloC$ does not have a near-unanimity operation of arity $k$, then $\Inv(\cloC)$ has an essential relation $R$ of arity $k$. 
Furthermore, this essential relation can be represented as a conjunction of critical relations and the arity of at least one of them should also be $k$. Thus, 
every clone $\cloC$ not having a near-unanimity operation of arity $k$ preserves a critical relation of arity $k$.

\begin{Lem}\label{lem:majorityWithKey}
Let $\cloC$ be an idempotent clone over $E_n$, for some $n\geq 2$, such that 
\begin{enumerate}
    \item $\cloC\models\Sigma_p$, for every prime $p\leq n$, and
    \item $\cloC\models\Sigma_{\mathrm{M}}$, i.e., $\cloC$ has a Mal'cev operation.
\end{enumerate}
Then $\cloC$ has a majority operation.
\end{Lem}
\begin{proof}
Let $\cloC$ be a clone satisfying all the hypotheses and suppose that $\cloC$ does not have a ternary near-unanimity operation, i.e., a majority operation. Then by Theorem~\ref{thm:BakerPixley} we have that $\Inv(\cloC)$ has a critical relation $R$ of arity $k\geq 3$. Therefore, by Theorem~\ref{thm:DimaKeyCritical}, for every nontrivial block $\textbf{B}$ of $R$, there exists an abelian group $\mathbf{G}=(G,+,-,0)$ whose order $\ell\leq n$ is the power of some prime 
and surjective mappings $\phi_i\colon B_i\to G$, for $i=1,2,\dots,k$ such that $R\cap\textbf{B} = \{(x_1,\dots,x_k)\mid \phi_1(x_1) + \phi_2(x_2) +\ldots +\phi_k(x_k) = 0\}$.

Let us show that the relation $R$ cannot be preserved by a cyclic operation $c_{p}$ of arity $p$, where $p$ divides $\ell$.
Choose a mapping $\psi_{i}\colon G\to B_{i}$, for every $i$,
such that $\phi_{i}(\psi_{i}(x)) = x$, for every $x\in G$.
Let $a$ be an element in $G$ of order $p$.
Notice, that 
\[B_{1}(x_1) = 
\exists x_2 \dots \exists x_k\;  \bigwedge_{i=2}^{k} R(x_1,\psi_{2}(0),\dots,\psi_{i-1}(0),x_i,\psi_{i+1}(0),\dots,\psi_{k}(0)),\]
which means that $B_{1}$ is pp-definable from $R$ and constants. Combining this with the idempotency of $\cloC$ we derive that 
the cyclic operation $c_p$ preserves $B_{1}$.
Similarly, we show that $c_{p}$ preserves $B_{i}$, for every $i$.
Applying 
$c_{p}$ to the rows of the matrices 
\[\begin{pmatrix}
\psi_{1}(0) &\psi_1(a)&\psi_{1}(2a)&\dots&\psi_{1}((p-1)a)\\
\psi_{2}(0) &\psi_2(-a)&\psi_{2}(-2a)&\dots&\psi_{2}(-(p-1)a)\\
\psi_{3}(0) &\psi_3(0)&\psi_{3}(0)&\dots&\psi_{3}(0)\\
\vdots&\vdots&\vdots&\ddots&\vdots\\
\psi_{k}(0) &\psi_k(0)&\psi_{k}(0)&\dots&\psi_{k}(0)
\end{pmatrix}\]
\[\begin{pmatrix}
\psi_{1}(0) &\psi_1(a)&\dots&\psi_{1}((p-2)a)&\psi_{1}((p-1)a)\\
\psi_{2}(-a) &\psi_2(-2a)&\dots&\psi_{2}(-(p-1)a)&\psi_{2}(0)\\
\psi_{3}(a) &\psi_3(a)&\dots&\psi_{3}(a)&\psi_{3}(a)\\
\psi_{4}(0) &\psi_4(0)&\dots&\psi_{4}(0)&\psi_{4}(0)\\
\vdots&\vdots&\vdots&\vdots&\vdots\\
\psi_{k}(0) &\psi_k(0)&\dots&\psi_{k}(0)&\psi_{k}(0)
\end{pmatrix}\]
we get respectively the tuples 
\begin{align*}
    &(c,d,\psi_{3}(0),\psi_{4}(0),\dots,\psi_{k}(0)), \text{ and }
    \\&(c,d,\psi_{3}(a),\psi_{4}(0),\dots,\psi_{k}(0))
\end{align*}
from $R\cap \textbf{B}$,
which contradicts the definition of $R\cap \textbf{B}$.
This contradiction proves that such a relation $R$ cannot exist in $\Inv(\cloC)$, thus $\cloC$ has a majority operation.
\end{proof}
\begin{Rem}
The result presented in Lemma~\ref{lem:majorityWithKey} can be alternatively proved using an argument coming from Tame Congruence Theory (see~\cite{HobbyMcKenzie}). One can show that, assuming the same hypotheses as in Lemma~\ref{lem:majorityWithKey}, a sufficient condition for the existence of a majority operation provided in \cite{HagemannHerrArithmetical} holds.
\end{Rem}

\begin{Cor}\label{cor:SHasMaj}
Let $\cloS$ be an idempotent clone over $E_3$ such that $\cloS\ \nleqm\ \cloC_2$, $\cloS\ \nleqm\ \cloC_3$, and $\cloS\ \nleqm\ \cloB_2$. Then $\cloS$ has a symmetric majority operation.
\end{Cor}
\begin{proof}
From Theorem \ref{thm:splittingMalcev} it follows that $\cloS$ has a Mal'cev operation, and from Theorem~\ref{thm:splittingCycles} it follows that there exist $c_2,c_3\in\cloC$ such that $c_2\models\Sigma_2$ and $c_3\models\Sigma_3$. Thus, it follows from Lemma~\ref{lem:majorityWithKey} that $\cloS$ has a majority operation $M'$. We define the operation $M$ as follows:
\begin{align}\nonumber\label{eq:majFS}
    M(x,y,z)\coloneqq c_2(&c_3(M'(x,y,z),M'(y,z,x),M'(z,x,y)),
    \\&c_3(M'(x,z,y),M'(z,y,x),M'(y,x,z))).\tag{$\heartsuit$}
\end{align}
It is easy to check that $M$ is a symmetric majority operation.
\end{proof}

\begin{Lem}\label{lem:SHasMinority}
Let $\cloS$ be an idempotent clone over $E_3$ such that $\cloS\ \nleqm\ \cloC_2$, $\cloS\ \nleqm\ \cloC_3$, and $\cloS\ \nleqm\ \cloB_2$. Then $\cloS$ has a symmetric minority operation.
\end{Lem}
\begin{proof}
Let $\cloS$ be as in the hypothesis. It follows from Theorem \ref{thm:splittingMalcev} that $\cloS$ has a Mal'cev operation $d$. Also, from Corollary~\ref{cor:SHasMaj} we know that $\cloS$ has a majority operation $M$. We define $m'_3$ as follows:
\begin{align*}
    m'_3(x,y,z)\coloneqq M(d(x,y,z),d(y,z,x),d(z,x,y)).
\end{align*}
It is easy to check that $m'_3$ is indeed a minority operation: note that, since $d$ is a Mal'cev operation, whenever we identify two variables in $m'_3$ at least two of the values among $d(x,y,z)$, $d(y,z,x)$, and $d(z,x,y)$ are equal to the variable that occurs only once. Hence, applying $M$ we obtain this variable, again. Furthermore, it follows from Theorem \ref{thm:splittingCycles} that $\cloS$ has a binary cyclic operation $c_2$ and a ternary cyclic operation $c_3$. We then define a symmetric minority $m_3$ in the same way we obtained a symmetric majority in Corollary \ref{cor:SHasMaj}: we simply replace every occurrence of $M$ in \eqref{eq:majFS} by $m_3$.
\end{proof}

\begin{Rem}\label{rem:symmetricAndConstant}
Note that the value of a symmetric minority $m_3(x,y,z)$ has to be a constant $c\in\{0,1,2\}$ whenever the three values in the scope of $m_3$ are all distinct, i.e.,
\begin{align*}
    m_3(0,1,2)&=m_3(0,2,1)=m_3(1,0,2)=\\&=m_3(1,2,0)=m_3(2,0,1)=\\&=m_3(2,1,0)= c.
\end{align*}
In this case we also denote the symmetric minority operation by $m_3^c$. We follow the same convention for symmetric majority operations.
\end{Rem}

\subsection{Generalized minority 
operations}\label{sec:oddition}
Recall that a \emph{generalized minority of arity $n$} is an $n$-ary operation satisfying the minor condition $\operatorname{GM}(n)$ from Definition~\ref{def:symmMinorConditions}. If a generalized minority $m_{n}$
is idempotent and only $a_{i}$
occurs an odd number of times in 
the tuple $(a_1,\dots,a_n)$, then 
$m_{n}(a_1,\dots,a_n)=a_i$.
Moreover, if $m_n$ is a generalized minority on $E_3$, then $m_n(a_1,\dots,a_n)$ returns a constant $c\in E_3$ on all the other tuples, that is on the tuples
containing an odd number of each element from~$E_{3}$.

Note that the minority operation $m_3(x,y,z)= x \oplus y \oplus z$ on the set $E_2=\{0,1\}$ is indeed a generalized minority of arity 3. Also note that if a clone $\cloC$ over $\{0,1\}$ contains the minority operation $m_3(x,y,z)$ then, for every $n\geq 2$, the generalized minority
\begin{align*}
    m_{2n+1}(x_1,\dots,x_{2n+1})&\coloneqq m_3(m_{2n-1}(x_1,\dots,x_{2n-1}),x_{2n},x_{2n+1})
    \\&= x_1 \oplus x_2 \oplus\ldots \oplus x_{2n+1}
\end{align*}
is also in $\cloC$. We prove an analogous result for the three-element case: we show that every clone $\cloS$ over $E_3$ satisfying condition~\eqref{eq:submax} has a generalized minority of every odd arity. Recall that Theorem~\ref{lem:SHasMinority} implies that $\cloS$ has a symmetric minority operation $m_3^c$ where $c\in\{0,1,2\}$ is some constant value that $m_3^c(x,y,z)$ returns whenever $|\{x,y,z\}|=3$, see Remark~\ref{rem:symmetricAndConstant}. We denote by $+_3$ the addition of the group of integers modulo $3$ and define the following auxiliary operation
\begin{equation}\tag{$\diamondsuit$}\label{eq:DimaSwitch}
    \small D^c(x,y,z)\coloneqq
    \begin{cases}
    c+_3 1 & \text{ if } (x,y,z)\in\{(c+_3 2,c,c+_3 1),(c+_3 2,c+_3 1,c)\},
    \\c+_3 2 & \text{ if } (x,y,z)\in\{(c+_3 1,c,c+_3 2),(c+_3 1,c+_3 2,c)\},
    \\x & \text{ otherwise.}
    \end{cases}
\end{equation}
Note that $D^c(x,y,z) = m_3^c(m_3^c(x,y,z),y,z)$, hence $D^c(x,y,z)\in\cloS$.

\begin{Thm}\label{thm:GenMinority}
Let $\cloS$ be an idempotent clone over $E_3$ such that $\cloS\ \nleqm\ \cloC_2$, $\cloS\ \nleqm\ \cloC_3$, and $\cloS\ \nleqm\ \cloB_2$. Then $\cloS$ has generalized minorities of arity $k$ for every odd $k\geq 3$. 
\end{Thm}
\begin{proof}
From Lemma~\ref{lem:SHasMinority} we know that $\cloS$ has a symmetric minority operation $m_3^c$; let $D^c$ the operation defined as in \eqref{eq:DimaSwitch}. For every $n\geq 2$, we define the operation
\begin{align*}
    m_{2n+1}^c(x_1,\dots,x_{2n+1}) \coloneqq m_3^c(&t_{2n+1}(x_1,x_2,x_3,x_4,\dots,x_{2n+1}),
        \\&t_{2n+1}(x_2,x_1,x_3,x_4,\dots,x_{2n+1}),
        \\&t_{2n+1}(x_3,x_1,x_2,x_4,\dots,x_{2n+1}))
\end{align*}
where
\begin{align*}
    t_{2n+1}(x_1,\dots,x_{2n+1})\coloneqq m_3^c(&D^c(m_{2n-1}^c(x_1,x_4,x_5,\dots,x_{2n+1}),x_1,x_1),
    \\&D^c(m_{2n-1}^c(x_1,x_4,x_5,\dots,x_{2n+1}),x_1,x_2),
    \\&D^c(m_{2n-1}^c(x_1,x_4,x_5,\dots,x_{2n+1}),x_1,x_3)).
\end{align*}
Note that the first argument of $m_3^c$ in the latter formula is always equal to \[m_{2n-1}^c(x_1,x_4,x_5,\dots,x_{2n+1})\]
however, for the sake of symmetry, we instead prefer to write \[D^c(m_{2n-1}^c(x_1,x_4,x_5,\dots,x_{2n+1}),x_1,x_1)\] in the definition. 

We are going to prove the claim of the theorem by induction over $n$. Let us first make a few remarks on the symmetries of $m^c_{2n+1}$ in order to make the formula more digestible for the reader. As an inductive hypothesis we assume that $m^{c}_{2n-1}$ is a generalized minority. 
It follows from the symmetry of $m^{c}_{2n-1}$ that 
$t_{2n+1}$ is invariant under any permutation of the variables 
$x_4,\dots,x_{2n+1}$.
Hence $m_{2n+1}^{c}$ is also invariant under any permutation of the variables 
$x_4,\dots,x_{2n+1}$.
Since $m_3^c$ is symmetric,
$t_{2n+1}$ is invariant under permutation of $x_2$ and $x_3$ and therefore 
$m_{2n+1}^{c}$ is invariant under any permutation of the variables 
$x_1$, $x_2$, and $x_3$. Notice that $m_1^c(x):=x$. Moreover, since $m_{2n-1}^c$ is a generalized minority, it holds \[m_{2n-1}^c(x_1,x_2,\dots,x_{2n-3},x,x) = m_{2n-3}^c(x_1,x_2,\dots,x_{2n-3}),\]
thus, we obtain that \[t_{2n+1}(x_1,x_2,\dots,x_{2n-1},x,x) = t_{2n-1}(x_1,x_2,\dots,x_{2n-1}),\]
and therefore\
$m_{2n+1}^c(x_1,x_2,\dots,x_{2n-1},x,x) = m_{2n-1}^c(x_1,x_2,\dots,x_{2n-1})$.

Combining this with the symmetry of $m_{2n+1}^c$ over permutation of the
last $2n-2$ coordinates,
we obtain that 
$m_{2n+1}^{c}$ behaves as a generalized minority for all the tuples 
having repetitive elements in 
$x_{4},\dots,x_{2n+1}$.
Thus, if $n\ge 3$ then $2n+1-3>3$ and 
$m_{2n+1}^{c}$ is a generalized minority.

Let us check how the identification of variables 
transforms $m_{5}^{c}$.
\begin{align*}\label{eq:odditionOne}
    m_{5}^c(x_1,x,x,x_4,x_5) = m_3^c\big(&t_{5}(x_1,x,x,x_4,x_5),\tag{$\star_1$}
        \\&t_{5}(x,x_1,x,x_4,x_5),
        \\&t_{5}(x,x_1,x,x_4,x_5)\big)
        \\ =\;\;t_{5}(&x_1,x,x,x_4,x_5)
     \\=m_3^c(&D^c(m_{3}^c(x_1,x_4,x_5),x_1,x_1),
    \\&D^c(m_{3}^c(x_1,x_4,x_5),x_1,x),
    \\&D^c(m_{3}^c(x_1,x_4,x_5),x_1,x))
    \\=D^c(&m_{3}^c(x_1,x_4,x_5),x_1,x_1)
    \\= m_{3}^c(&x_1,x_4,x_5)
\end{align*}
This proves that
$m_{5}^{c}$ behaves well on all the tuples having repetitive elements in 
the first 3 coordinates. It only remains to consider
the case when there are no repeated elements in the first three components
and no repeated elements in the last two components of $m^c_5$. By making use of
the known symmetries, it suffices to verify that $m^c_5(x_1, x_2 , x_3 , x_1 , x_2 ) = x_3$.
\begin{align*}
    t_{5}(x_1,x_2,x_3,x_1,x_2)= 
    m_3^c\big(&D^c(m_{3}^c(x_1,x_1,x_2),x_1,x_1),
    \\&D^c(m_{3}^c(x_1,x_1,x_2),x_1,x_2),
    \\&D^c(m_{3}^c(x_1,x_1,x_2),x_1,x_3)\big)
    \\ = m_3^c(&x_2,x_2,D^c(x_2,x_1,x_3))=D^c(x_2,x_1,x_3);\\
    t_{5}(x_1,x_2,x_3,x_2,x_3)= 
    m_3^c(&D^c(m_{3}^c(x_1,x_2,x_3),x_1,x_1),
    \\&D^c(m_{3}^c(x_1,x_2,x_3),x_1,x_2),
    \\&D^c(m_{3}^c(x_1,x_2,x_3),x_1,x_3))=m_3^c(x_1,x_2,x_3).
\end{align*}
We check the last equality as follows.
If $|\{x,y,z\}|<3$, then $D^{c}(x,y,z) = x$. Hence if $|\{x_1,x_2,x_3\}|<3$, then $D^c(m_{3}^c(x_1,x_2,x_3),x_1,x_3) = m_{3}^c(x_1,x_2,x_3)$.
If $|\{x_1,x_2,x_3\}|=3$, then 
$m_{3}^c(x_1,x_2,x_3)=c$ and, since $D^{c}$ returns $c$ whenever the first coordinate is $c$, we get the equality.

Finally, we obtain the following equation
\begin{align*}\label{eq:odditionThree}
    m_{5}^c(x_1,x_2,x_3,x_1,x_2) 
    = m_3^c\big(&t_{5}(x_1,x_2,x_3,x_1,x_2),\tag{$\star_2$}
        \\&t_{5}(x_2,x_1,x_3,x_1,x_2),
        \\&t_{5}(x_3,x_1,x_2,x_1,x_2)\big)
    \\= m_3^c\big(&D^{c}(x_2,x_1,x_3),
        \\&D^{c}(x_1,x_2,x_3),
        \\&m_{3}^{c}(x_1,x_2,x_3)\big)= x_3.
\end{align*}

The equation \eqref{eq:odditionThree} above can be checked manually.
If $\{x_1,x_2,x_3\}\neq\{0,1,2\}$, then it again follows from 
the fact that $m_{3}^{c}$ is the minority and $D^{c}$ is the first projection on every 2-element subset.

If $\{x_1,x_2,x_3\}=\{0,1,2\}$ and $x_{3} = c$, then 
\[m_3^c(D^{c}(x_2,x_1,x_3),D^{c}(x_1,x_2,x_3),m_{3}^{c}(x_1,x_2,x_3))=
m_3^c(x_1,x_2,c)=c.\]

If $\{x_1,x_2,x_3\}=\{0,1,2\}$ and $x_{1} = c$, then 
\[m_3^c(D^{c}(x_2,x_1,x_3),D^{c}(x_1,x_2,x_3),m_{3}^{c}(x_1,x_2,x_3))=
m_3^c(x_3,c,c)=x_3.\]
Similarly, if $\{x_1,x_2,x_3\}=\{0,1,2\}$ and $x_{2} = c$, then 
\[m_3^c(D^{c}(x_2,x_1,x_3),D^{c}(x_1,x_2,x_3),m_{3}^{c}(x_1,x_2,x_3))=
m_3^c(c,x_3,c)=x_3.\]
The equations \eqref{eq:odditionOne} and \eqref{eq:odditionThree} imply that 
$m_{5}^{c}$ is a generalized minority.
\end{proof}

\subsection{Totally symmetric operations of every arity}

Here we prove that every idempotent clone $\cloS$ over $E_3$ satisfying the condition~\eqref{eq:submax} has totally symmetric operations of every arity $n\geq 2$ (see Definition~\ref{def:symmMinorConditions}).  
\begin{Thm}\label{thm:TSn}
Let $\cloS$ be an idempotent clone over $E_3$ such that $\cloS\ \nleqm\ \cloC_2$, $\cloS\ \nleqm\ \cloC_3$, and $\cloS\ \nleqm\ \cloB_2$. Then $\cloS$ has a totally symmetric operation $s_n$ of arity $n$, for every $n\geq 2$.
\end{Thm}
\begin{proof}
From Corollary~\ref{cor:SHasMaj} and Lemma~\ref{lem:SHasMinority} it follows that $\cloS$ has a symmetric majority operation $M^c$ and a symmetric minority operation $m$, respectively. Also, from Theorem~\ref{thm:splittingCycles} it follows that there exists a binary cyclic operation $s_2\in\cloS$, thus $\cloS\models\operatorname{TS}(2)$. For every $n\geq 3$ we define:
\begin{align*}
    s_n(x_1,\dots,x_n) \coloneqq m(&s_{n-1}(x_1,M^c(x_1,x_2,x_3),x_4,\dots,x_n),
    \\&s_{n-1}(x_2,M^c(x_1,x_2,x_3),x_4,\dots,x_n),
    \\&s_{n-1}(x_3,M^c(x_1,x_2,x_3),x_4,\dots,x_n)).
\end{align*}
We will prove by induction on $n\ge 2$ that 
\begin{enumerate}[(i)]
    \item if $\{x_1,\dots,x_n\}=\{a,b\}\subset\{0,1,2\}$, then $s_n(x_1,\dots,x_n)=s_2(a,b)$;
    \item if $\{x_1,\dots,x_n\}=\{0,1,2\}$, then \[s_n(x_1,\dots,x_n)= 
    m(s_2(0,c),s_2(1,c),s_2(2,c)).\]
\end{enumerate}
For $n=2$ this is obvious. Notice that,
for every $n\ge 3$,
\begin{align*}
    s_n(x,x,x_3,x_4,\dots,x_n) \coloneqq m(&s_{n-1}(x,M^c(x,x,x_3),x_4,\dots,x_n),
    \\&s_{n-1}(x,M^c(x,x,x_3),x_4,\dots,x_n),
    \\&s_{n-1}(x_3,M^c(x,x,x_3),x_4,\dots,x_n))\\
    &\;\;\;=s_{n-1}(x_3,x,x_4,\dots,x_n)
\end{align*}
Hence, by the inductive assumption we have the required properties (i) and (ii) 
on all tuples whose first two elements are equal.
Since the operations $M^{c}$ and $m$
are symmetric, $s_{n}$ is symmetric under any permutation 
of the first 3 variables.
Therefore, the 
property (i) always holds and the property (ii)
holds on all tuples such that the first three elements are not different.

Let us prove the property (ii) on all tuples 
$(x_1,x_2,\dots,x_n)$ such that 
$\{x_1,x_2,x_3\} = \{0,1,2\}$.
For $s_3$ it immediately follows from the definition. 
To prove this for $n>3$
consider 3 cases.

Case 1. If $\{x_4,\dots,x_n\}=\{a\}\subset\{0,1,2\}$ then
\begin{align}\label{eq:TStwo}
    s_n(x_1,\dots,x_n) = 
    m(&s_{n-1}(x_1,c,a,\dots,a),\notag
    \\&s_{n-1}(x_2,c,a,\dots,a),\notag
    \\&s_{n-1}(x_3,c,a,\dots,a))\notag
    \\\stackrel{\star}{=} m(&s_{n-1}(0,c,a,\dots,a),\tag{$\bullet_1$}
    \\&s_{n-1}(1,c,a,\dots,a),\notag
    \\&s_{n-1}(2,c,a,\dots,a)).\notag
\end{align}
The equality $\stackrel{\star}{=}$ holds because $m$ is symmetric. In case $a=c$, we obtain, by the induction hypothesis, that \[s_n(x_1,\dots,x_n)=m(s_2(0,c),s_2(1,c),s_2(2,c)).\] 
If $a\neq c$, then in \eqref{eq:TStwo} we have an argument of the form $s_{n-1}(a,c,a,\dots,a)$, one of the form $s_{n-1}(c,c,a,\dots,a)$, and one where 0,1, and 2 occur. Therefore, by property (i) of the induction hypothesis we get $s_{n-1}(a,c,a,\dots,a)=s_2(a,c)$ and $s_{n-1}(c,c,a,\dots,a)=s_2(a,c)$. Moreover, by properties of $m$, we get \[s_n(x_1,\dots,x_n)=m(s_2(a,c),s_2(a,c),s_3(0,1,2))=s_3(0,1,2).\]  

Case 2. If $\{x_4,\dots,x_n\}=\{a,b\}\subset\{0,1,2\}$ then, by using the fact that $s_{n-1}$ and $m$ are symmetric, we get
\begin{align}\label{eq:TSthree}
    s_n(x_1,\dots,x_n) = m(&s_{n-1}(0,c,a,\dots,a,b,\dots,b), \tag{$\bullet_2$}
    \\&s_{n-1}(1,c,a,\dots,a,b,\dots,b),\notag
    \\&s_{n-1}(2,c,a,\dots,a,b,\dots,b)).\notag
\end{align}
If $c\notin\{a,b\}$ then each argument of $m$ in the latter formula is equal to $s_3(0,1,2)$, by the induction hypothesis. Otherwise, if $c\in\{a,b\}$ then in \eqref{eq:TSthree} we have an argument of the form $s_{n-1}(a,\dots,a,b,\dots,b)$, one of the form $s_{n-1}(b,a,\dots,a,b\dots,b)$, and one where 0, 1, and 2 occur. By the induction hypothesis, we get
\[s_n(x_1,\dots,x_n)=m(s_2(a,b),s_2(a,b),s_3(0,1,2)).\]

Case 3. If $\{x_4,\dots,x_n\}=\{0,1,2\}$, then, by the symmetry of $m$ and $s_{n-1}$, we have
\begin{align*}\label{eq:TSsurj}
    s_n(\sigma(x_1),\dots,\sigma(x_n)) = m(&s_{n-1}(0,c,0,\dots,0,1,\dots,1,2\dots,2), \tag{$\bullet_3$}
    \\&s_{n-1}(1,c,0,\dots,0,1,\dots,1,2\dots,2),\notag
    \\&s_{n-1}(2,c,0,\dots,0,1,\dots,1,2\dots,2))\notag
\end{align*}
It follows, by the the induction hypothesis, that each argument of $m$ in \eqref{eq:TSsurj} is equal to $s_3(0,1,2)$; hence, we obtain $s_n(x_1,\dots,x_n)=s_3(0,1,2)$.
This concludes the proof.
\end{proof}

\subsection{The main result}
In Section~\ref{sec:coatom} we proved that $\overline{\cloI_2}$ is the unique coatom in $\Pthree$. 
Here we prove that, whenever a clone has totally symmetric operations and generalized minorities of an arbitrary large arity,
there exists a minor homomorphism from $\mathcal I_2$ to this clone. Combining this with the results of the previous sections
we derive the main result of our paper: $\overline{\cloC_2}$, $\overline{\cloC_3}$, and $\overline{\cloB_2}$ are the only submaximal elements in $\Pthree$.

It is well known that every operation 
over $\{0,1\}$ has a unique polynomial representation if we forbid 
repetitive monomials and disrespect the order of monomials. 
Applying this fact to idempotent operations from $\cloI_2$
we obtain the following lemma,
in which 
operations $\oplus$ and $\wedge$ denote the usual sum and multiplication modulo 2, respectively.

\begin{Lem}\label{Lem:polynomialRepresentation}
For every operation $f\in \cloI_2$ there exists an up to the order of monomials unique
representation of the form $f(x_1,\dots,x_n)\coloneqq\bigoplus_{i=1}^\ell \bigwedge W_i$,
where $\ell$ is odd and the sets $W_1,\dots,W_l\subseteq\{x_1,\dots,x_n\}$ are different and nonempty.
\end{Lem}

\begin{proof}
It is sufficient to check that every polynomial preserving 
$\{0\}$ does not have the constant 1 as a monomial, 
and that every polynomial preserving $\{1\}$ has an odd number of monomials.
\end{proof}

\begin{Thm}\label{thm:new}
Let $\cloS$ be a clone over $E_k$, for some $k\geq 2$, such that 
\begin{itemize}
    \item $\cloS\models\operatorname{TS}(n)$, for every $n\geq 2$, and
    \item $\cloS\models\GM(n)$, for every odd $n\ge 3$.
\end{itemize}
Then there exists a minor homomorphism from $\cloI_2$ to $\cloS$.
\end{Thm}
\begin{proof}
Let $f$ be any operation in $\cloI_2$.
Notice that the identification of two variables of a totally symmetric operation of arity $n$ gives a totally symmetric operation of a smaller arity. 
Similarly, the identification of three variables of 
a generalized minority gives a generalized minority of a smaller arity.
Then by K\"onig's lemma there exist an infinite sequence of totally symmetric operations $s_{2},s_{3},s_{4},\dots,$ and an infinite sequence of generalized minorities $m_{3},m_5,m_7,\dots,$ such that $s_{n}$ and $m_n$ are of arity $n$ for every $n$, and they are compatible in the following sense.
The identification of two variables of $s_{n}$ gives $s_{n-1}$
and the identification of three variables of $m_{2k+1}$ gives
$m_{2k-1}$.

By Lemma~\ref{Lem:polynomialRepresentation} 
there exists an up to permutation of monomials unique representation $f(x_1,\dots,x_k)=\bigoplus_{i=1}^\ell \bigwedge W_i$, where $\ell$ is odd and the sets $W_1,\dots,W_l\subseteq\{x_1,\dots,x_n\}$ are different. Notice that, for every $i\geq 2$, the operation $s_i$ only depends on the set of variables occurring in it, i.e., the order of the variables and their multiplicity can be ignored. Thus, we write $s_{|W_i|}(W_i)$ to stress this fact; moreover, we set $s_1(\{x\})\coloneqq x$, for every $x\in\{x_1,\dots,x_k\}$. We define the map $\xi\colon \cloI_2\to\cloS$ as follows
\[\xi\colon \Big(\bigoplus_{i=1}^\ell \bigwedge W_i\Big)\mapsto m_\ell\big(s_{|W_1|}(W_1),\dots,s_{|W_\ell|}(W_\ell)\big).\]
Since $m_{\ell}$ is symmetric, the map $\xi$ is well defined. 

Note that both the operation $\oplus$ and $m_\ell$ only depend on the parity of the elements occurring among their arguments.

Let $\pi\colon\{1,\dots,k\}\to\{1,\dots,r\}$ be a map. By first applying the map $\xi$ we obtain $m_\ell\big(s_{|W_1|}(W_1),\dots,s_{|W_\ell|}(W_\ell)\big)$ and then, via $\pi$, we obtain 
\begin{equation*}
    m_\ell\big(s_{|W^\pi_1|}(W^\pi_1),\dots,s_{|W^\pi_{\ell}|}(W^\pi_{\ell})\big), \mbox{ where } W^\pi_{i}\coloneqq\{x_{\pi(j)}\mid x_j\in W_i\}.
\end{equation*}
Let 
$\{U_1,\dots,U_{t}\}$ be the set of all different subsets in 
$\{W^\pi_1,\dots,W^\pi_\ell\}$.
Without loss of generality we assume that 
$U_{i}$ appears an odd number of times in $W^\pi_1,\dots,W^\pi_\ell$
for $i\in\{1,2\dots,d\}$ and 
$U_{i}$ appears an even number of times in $W^\pi_1,\dots,W^\pi_\ell$
for $i\in\{d+1,d+2\dots,t\}$.
Then using properties of $m_{\ell}$ we have 
\[m_\ell\big(s_{|W^\pi_1|}(W^\pi_1),\dots,s_{|W^\pi_{\ell}|}(W^\pi_{\ell})\big)
= 
m_d\big(s_{|U_1|}(U_1),\dots,s_{|U_{d}|}(U_{d})\big).
\]

On the other side, if we first apply $\pi$, we get $\bigoplus_{i=1}^\ell \bigwedge W^\pi_{i} = \bigoplus_{i=1}^d \bigwedge U_{i}$.
Since all the monomials in
$\bigoplus_{i=1}^d \bigwedge U_{i}$ are different, $\xi$ applied to it gives us 
\[m_d\big(s_{|U_1|}(U_1),\dots,s_{|U_{d}|}(U_{d})\big),\] which is exactly what we need.
\end{proof}

\begin{Cor}\label{thm:main}
Let $\cloS$ be an idempotent clone over $E_3$ such that $\cloS\ \nleqm\ \cloC_2$, $\cloS\ \nleqm\ \cloC_3$, and $\cloS\ \nleqm\ \cloB_2$. There is a minor homomorphism from $\cloI_2$ to $\cloS$.
\end{Cor}
\begin{proof}
From Theorem~\ref{thm:GenMinority} we know that $\cloS$ has a generalized minority $m_\ell$, for every odd number $\ell\ge 3$. Moreover, from Theorem~\ref{thm:TSn} it follows that $\cloS$ has a totally symmetric operation $s_n$ of arity $n$, for every $n\geq 2$.
Thus, the claim follow from Theorem~\ref{thm:new}.
\end{proof}

\section{Conclusion}

The results presented in this article together with the work from~\cite{BodirskyVucajZhuk} give hope that a complete description of $\Pthree$ might be achievable. We conclude this article by stating three open problems, with the aim of suggesting a path leading to a full description of $\Pthree$.
\begin{problem}\label{prob:atoms}
Find all the atoms of $\Pthree$. Is every atom of $\Pthree$ of the form $\overline{\cloC}$, for some finitely related clone $\cloC$?
\end{problem}
Note that a positive solution to Problem~\ref{prob:atoms} would provide a concrete list of the hardest tractable CSPs over $\{0,1,2\}$, refining \cite{Bulatov-3-conf}.

By Corollary~\ref{thm:main} it immediately follows that $\overline{\cloC_2}$, $\overline{\cloC_3}$, and $\overline{\cloB_2}$ are exactly the submaximal elements of $\Pthree$.
Furthermore, in order to prove that $\Pthree$ has at most countably infinite many elements, we can even focus only on those clones for which there exists a minor homomorphism to $\cloB_2$. Indeed, we prove that every clone $\cloC$ over a finite set has a \emph{Mal'cev operation} (Theorem~\ref{thm:splittingMalcev}) if and only if there is no minor homomorphism from $\cloC$ to $\cloB_2$ and Bulatov~\cite{BulatovMalcev} proved that there are only finitely many clones over $\{0,1,2\}$ containing a Mal'cev operation.

\begin{problem}\label{prob:Malcev}
Find all elements of $\Pthree$ that are below $\overline{\cloC_2}$.
\end{problem}
It follows from Theorems~\ref{thm:main} and~\ref{thm:splittingMalcev} that all clones over $\{0,1,2\}$ with a Mal'cev operation are below $\overline{\cloC_2}$ or $\overline{\cloC_3}$ in $\Pthree$. Since all the elements of $\Pthree$ which are below $\overline{\cloC_3}$ were found in \cite{BodirskyVucajZhuk}, in order to solve Problem~\ref{prob:Malcev}, we have to consider all clones over $\{0,1,2\}$ with a Mal'cev operation and a cyclic operation of arity 3, and order them with respect to $\leqm$.
\begin{problem}\label{prob:B2}
Find all elements of $\Pthree$ that are below $\overline{\cloB_2}$.
\end{problem}
We would like to emphasise that, out of the three problems proposed in this section, Problem~\ref{prob:B2} is the more challenging one as, a priori, there might exist continuum many elements below $\overline{\cloB_2}$.

An alternative direction that this research strand can take is to investigate whether Corollary~\ref{thm:main} can be generalized to domains strictly larger than three. More precisely, one could ask if the following holds: for every $n>3$, if $\cloS$ is an idempotent clone over $E_n$ such that, for every prime $p\leq n$, it holds that $\cloS\ \nleqm\ \cloC_p$ and $\cloS\ \nleqm\ \cloB_2$, then there exists a minor homomorphism from $\cloI_2$ to $\cloS$. The latter statement was proved to be true if we only consider polymorphism clones of finite directed graphs~\cite{BodirskyStarke}. 

However, there is a structure in the literature that proves the statement to be false in general. Carvalho and Krokhin~\cite{CarvalhoKrokhin} -- for different purposes -- presented a structure $\mathbb{K}$ with 21 elements that has cyclic polymorphisms
of all arities, a Mal'cev polymorphism, and that does not have any symmetric polymorphism of arity 5. Note that, from Theorem~\ref{thm:splittingCycles} and $\Pol(\mathbb{K})\models\Sigma_p$ for every prime $p$, it follows that $\Pol(\mathbb{K})\ \nleqm\ \cloC_p$. Moreover, it follows from $\Pol(\mathbb{K})\models\Malcev$ and Theorem~\ref{thm:splittingMalcev} that $\Pol(\mathbb{K})\ \nleqm\ \cloB_2$. However, we have that $\cloI_2\ \nleqm\ \Pol(\mathbb{K})$, since $\mathcal{I}_2\models \operatorname{FS}(5)$ while $\Pol(\mathbb{K})\not\models \operatorname{FS}(5)$.

For the sake of full disclosure, the structure $\mathbb{K}$ is defined as follows: $\mathbb{K}\coloneqq (K; R, S)$, where $K = \{0,1,2,\dots,9,10,a,b,c,d,e,f,g,h,i,j\}$,
and $R$ and $S$ are binary relations that are the graphs of the following permutations $r$ and $s$, respectively (see Figure~\ref{fig:KrokhinCarvalho}),
\begin{align*}
    r &= (0\ 1\ 2)(5\ 6\ 7)(8\ 9\ 10)(e\ b\ a)(d\ g\ i)(f\ h\ c),
    \\s &= (1\ 4)(2\ 3)(5\ 6)(7\ 8)(j\ e)(b\ c)(a\ d)(i\ f).
\end{align*} 

\begin{figure}[H]
\begin{center}
\begin{tikzpicture}[scale=.7]
\draw[->][red] (-.1,3) arc [radius=.15, start angle=0,end angle=345];
\draw[->][blue] (1.6,4) arc [radius=.15, start angle=180,end angle=520];
\draw[->][blue] (2.1,3) arc [radius=.15, start angle=180,end angle=520];
\draw[->][red] (7.1,3) arc [radius=.15, start angle=180,end angle=520];
\draw[->][red] (6.5,4.1) arc [radius=.15, start angle=270,end angle=590];
\draw[red] (5,3) -- (6,3);
\draw[red] (4,3) to [bend left=40](4.5,4) ;
\draw[red] (.5,4) -- (1.5,4);
\draw [red] (1,3) -- (2,3);
\draw[->][blue] (0,3) -- (.45,3.95);
\draw[->][blue] (.5,4) -- (.95,3.05);
\draw[->][blue] (1,3) -- (.07,3);
\draw[->][blue] (4,3) -- (4.45,3.95);
\draw[->][blue] (4.5,4) -- (4.95,3.05);
\draw[->][blue] (5,3) -- (4.07,3);
\draw[->][blue] (6,3) -- (6.45,3.95);
\draw[->][blue] (6.5,4) -- (6.95,3.05);
\draw[->][blue] (7,3) -- (6.07,3);
\draw[fill] (0,3) circle [radius=0.05];
\draw[fill] (1,3) circle [radius=0.05];
\draw[fill] (2,3) circle [radius=0.05];
\draw[fill] (.5,4) circle [radius=0.05];
\draw[fill] (1.5,4) circle [radius=0.05];
\draw[fill] (4,3) circle [radius=0.05];
\draw[fill] (5,3) circle [radius=0.05];
\draw[fill] (4.5,4) circle [radius=0.05];
\draw[fill] (6,3) circle [radius=0.05];
\draw[fill] (7,3) circle [radius=0.05];
\draw[fill] (6.5,4) circle [radius=0.05];
\node[below] at (0,3) {\small 0};
\node[below] at (1,3) {\small 2};
\node[above] at (0.5,4) {\small 1};
\node[below] at (2,3) {\small 3};
\node[above] at (1.5,4) {\small 4};
\node[below] at (4,3) {\small 5};
\node[below] at (5,3) {\small 7};
\node[below] at (6,3) {\small 8};
\node[below] at (7,3) {\small 10};
\node[above] at (4.5,4) {\small 6};
\node[right] at (6.5,4) {\small 9};

\draw[->][blue] (-.1,0) arc [radius=.15, start angle=0,end angle=345];
\draw[->][red] (3.5,2.1) arc [radius=.15, start angle=270,end angle=590];
\draw[->][red] (4,-1.1) arc [radius=.15, start angle=-270,end angle=70];
\draw[red] (4,1) -- (4.5,0);
\draw[->][blue] (3,1) -- (3.92,1);
\draw[->][blue] (4,1) -- (3.55,1.95);
\draw[->][blue] (3.5,2) -- (3.05,1.05);
\draw [red] (0,0) -- (1,0);
\draw [red] (1.5,1) -- (3,1);
\draw[red] (2,0) -- (3.5,0);
\draw[->][blue] (1,0) -- (1.45,.95);
\draw[->][blue] (1.5,1) -- (1.95,.05);
\draw[->][blue] (2,0) -- (1.07,0);
\draw[->][blue]  (3.5,0) -- (3.95,-.95);
\draw[->][blue] (4,-1) -- (4.45,-.05);
\draw[->][blue] (4.5,0) -- (3.58,0);
\draw[fill] (0,0) circle [radius=0.05];
\draw[fill] (1,0) circle [radius=0.05];
\draw[fill] (2,0) circle [radius=0.05];
\draw[fill] (1.5,1) circle [radius=0.05];
\draw[fill] (3,1) circle [radius=0.05];
\draw[fill] (4,1) circle [radius=0.05];
\draw[fill] (3.5,2) circle [radius=0.05];
\draw[fill] (3.5,0) circle [radius=0.05];
\draw[fill] (4.5,0) circle [radius=0.05];
\draw[fill] (4,-1) circle [radius=0.05];

\node[below] at (0,0) {\small $j$};
\node[below] at (1,0) {\small $e$};
\node[below] at (2,0) {\small $a$};
\node[below] at (3,1) {\small $c$};
\node[above] at (1.5,1) {\small $b$};
\node[above] at (3.5,0) {\small $d$};
\node[right] at (4,1) {\small $f$};
\node[right] at (4.5,0) {\small $i$};
\node[right] at (4,-1) {\small $g$};
\node[right] at (3.5,2) {\small $h$};
\end{tikzpicture}
\caption{The structure $\mathbb{K}\coloneqq (K; {\color{blue}R}, {\color{red}S})$.}
\label{fig:KrokhinCarvalho}
\end{center}
\end{figure}

In the light of this, we conclude this article with the following conjecture.

\begin{Conj}\label{conj:submaximal}
Let $\cloS$ be a clone over $E_k$, for some $k\geq 2$, such that
\begin{itemize}
    \item $\cloS$ satisfies $\operatorname{TS}(n)$, for every $n\geq 2$, and
    \item $\cloS$ satisfies $\Malcev$.
\end{itemize} Then there exists a minor homomorphism from $\cloI_2$ to $\cloS$.
\end{Conj}

Note that Conjecture~\ref{conj:submaximal} implies Theorem~\ref{thm:new}: indeed, the assumption of having a quasi Mal'cev operation is a strictly weaker assumption than that of requiring the existence of a generalized minority of every odd arity.

\section*{Declarations}
Albert Vucaj is funded by the Austrian Science
Fund (FWF) [P 32337, I~5948]. For the purpose of Open Access, the author has
applied a CC BY public copyright licence to any Author Accepted
Manuscript (AAM) version arising from this submission. Albert Vucaj is also funded by the European Union (ERC, POCOCOP, 101071674). Views and
opinions expressed are however those of the author(s) only and do not
necessarily reflect those of the European Union or the European
Research Council Executive Agency. Neither the European Union nor the
granting authority can be held responsible for them.

Dmitriy Zhuk is funded by the European Union (ERC, POCOCOP, 101071674). Views and
opinions expressed are however those of the author(s) only and do not
necessarily reflect those of the European Union or the European
Research Council Executive Agency. Neither the European Union nor the
granting authority can be held responsible for them.

Conflict of interest: The authors declare that they have no relevant financial or non-financial interests to disclose.
\bibliographystyle{spmpsci.bst}
\def\cprime{$'$} \def\cprime{$'$}

\end{document}